\documentclass{amsart}
\usepackage{amsmath}
\usepackage{amsfonts}
\usepackage{amsthm}
\usepackage{amssymb}
\usepackage[all]{xy}
\usepackage{color}
\usepackage{eucal}
\SelectTips{cm}{}

\DeclareMathOperator{\Ho}{Ho}

\DeclareMathOperator{\Hom}{Hom}
\DeclareMathOperator{\Ext}{Ext}

\DeclareMathOperator{\sset}{sSet}
\DeclareMathOperator{\Ch}{Ch}

\DeclareMathOperator{\Sp}{Sp}
\DeclareMathOperator{\map}{map}

\DeclareMathOperator{\holim}{holim}

\DeclareMathOperator{\CAT}{CAT}
\DeclareMathOperator{\Sect}{Sect}
\DeclareMathOperator{\Post}{Post}
\DeclareMathOperator{\Bou}{Bou}
\DeclareMathOperator{\Tow}{Tow}
\DeclareMathOperator{\Chrom}{Chrom}
\DeclareMathOperator{\Fibprod}{Fibpr}
\DeclareMathOperator{\Hofib}{Fib}
\DeclareMathOperator{\fib}{\it fib}

\newcommand{\C}{\mathcal{C}}

\newcommand{\G}{\mathcal{G}}

\newcommand{\K}{\mathcal{K}}
\newcommand{\calS}{\mathcal{S}}

\newcommand{\fin}{\textit{fin}}
\newcommand{\arr}{{\rm Arr}}
\newcommand{\weq}{{\rm weq}}

\theoremstyle{plain}
\newtheorem{theorem}{Theorem}[section]
\newtheorem{proposition}[theorem]{Proposition}
\newtheorem{corollary}[theorem]{Corollary}
\newtheorem{lemma}[theorem]{Lemma}

\newtheorem*{theorem*}{Theorem}

\theoremstyle{definition}
\newtheorem{definition}[theorem]{Definition}

\theoremstyle{remark}

\newtheorem{rmk}[theorem]{Remark}

\newcommand{\lradjunction}{\,\,\raisebox{-0.1\height}{$\overrightarrow{\longleftarrow}$}\,\,}

\xyoption{all}

\begin{document}

\title[Towers and fibered products of model structures]{Towers and fibered products \\of model structures}
\author[J.J. Guti\'errez]{Javier J. Guti\'errez}
\address{Radboud Universiteit Nijmegen, Institute for
Mathematics, Astrophysics, and Particle Physics, Heyendaalseweg 135, 6525 AJ
Nijmegen, The Netherlands} \email{j.gutierrez@math.ru.nl}
\urladdr{http://www.math.ru.nl/~gutierrez}
\author[C. Roitzheim]{Constanze Roitzheim}
\address{University of Kent, School of Mathematics, Statistics and Actuarial
Science, Canterbury, Kent CT2 7NF, United Kingdom}
\email{C.Roitzheim@kent.ac.uk}
\urladdr{http://www.kent.ac.uk/smsas/personal/csrr}
\thanks{The first author was supported by the NWO (SPI 61-638) and the MEC-FEDER grants MTM2010-15831 and MTM2013-42178-P.
Both authors received support from the LMS Scheme~4 grant no. ~41360.}

\begin{abstract}
Given a left Quillen presheaf of localized model structures, we study the homotopy limit model structure on the associated category of sections. We focus specifically on towers and fibered products (pullbacks) of model categories. As applications we consider Postnikov towers of model categories, chromatic towers of spectra and Bousfield arithmetic squares of spectra. For stable model categories, we show that the homotopy fiber of a stable left Bousfield localization is a stable right Bousfield localization.
\end{abstract}

\maketitle

\section*{Introduction}
Localization techniques play an important role in modern homotopy theory. For
several applications it is often useful to approximate a given space or
spectrum by simpler ones by means of localization functors. For instance,
given a simplicial set $X$, one can consider its Postnikov tower.
This tower can be built as a sequence of fibrations
$$
\cdots\stackrel{f_{n}}{\longrightarrow} P_nX\stackrel{f_{n-1}}{\longrightarrow} P_{n-1}X\stackrel{f_{n-2}}{\longrightarrow}\cdots\stackrel{f_{2}}{\longrightarrow} P_2 X \stackrel{f_{1}}{\longrightarrow}
P_1X\stackrel{f_{0}}{\longrightarrow}  P_0X
$$
and maps $p_n\colon X\to P_n X$ satisfying that $p_n=f_n\circ p_{n+1}$ for
every $n\ge 0$, and that $\pi_k(f_n)\colon \pi_k(X)\cong \pi_k(P_nX)$ if
$k\le n$ for any choice of base point of $X$, and $\pi_k(P_nX)=0$ if $k>n$ and all choices of base points.

Each of the spaces $P_n X$ can be built as a localization of $X$ with respect to the map $S^{n+1}\to *$,
and $p_n$ is the corresponding localization map. If $X$ is connected, then the fiber of $f_{n-1}$ is an
Eilenberg--Mac\,Lane space $K(\pi_n(X), n)$ and every simplicial set $X$ can be
reconstructed as the homotopy limit of its Postnikov tower $X\simeq \holim_
{n\ge 0} P_nX$; see~\cite[Ch.VI, Theorem 3.5]{GJ99}.

In the category of spectra, given any spectrum $E$, we can consider its
associated homological localization functor $L_E$ which inverts the maps that
induce isomorphisms in $E_*$-homology in a universal way. Given an abelian
group $G$, let us denote by $MG$ the associated Moore spectrum. It is
well-known that any spectrum $X$ can be built, using Bousfield's arithmetic
square~\cite{Bou79}, as a homotopy pullback of the diagram of homological
localizations
\[
L_{M\mathbb{Z}_J}X \longrightarrow L_{M\mathbb{Q}}X \longleftarrow
L_{M\mathbb{Z}_K}X,
\]
where, $J$ and $K$ form any partition of the set of prime numbers and
$\mathbb{Z}_J$ are the integers localized at the set of primes $J$.

Furthermore, the chromatic convergence theorem \cite[Theorem 7.5.7]{Rav92}
states that a finite $p$-local spectrum $X$ is the homotopy limit of its
chromatic localizations $L_{E(n)}X$ at the prime $p$.

The aim of this paper is to present categorified versions of these statements
in the framework of Quillen model structures.  Given a diagram (left Quillen
presheaf) of model categories $F\colon \mathcal{I}^{\rm op}\to \CAT$, there
is an injective model structure on the category of sections associated to
$F$, which we can further colocalize in order to obtain the homotopy limit
model structure. We study these model structures for towers and homotopy
fibered products (homotopy pullbacks) of model categories.

Firstly, we construct the Postnikov tower of an arbitrary combinatorial model
category. As an application we show that for simplicial sets and for bounded
below chain complexes these towers converge in a certain sense. Another tower
model structure is the homotopy limit model structure on the left Quillen
presheaf of chromatic towers $\Chrom(\Sp)$, where $\Sp$ denotes here the category of $p$-local symmetric spectra. We show that the Quillen
adjunction
\[
{\rm const}: \Sp \lradjunction \Chrom(\Sp): \lim
\]
induces a composite
\[
\Ho(\Sp)^{\fin} \xrightarrow{\mathbb{L}{\rm const}} \Ho(\Chrom(\Sp))^F
\xrightarrow{\holim} \Ho(\Sp)^{\fin}
\]
which is isomorphic to the identity. (Here, $F$ and $\fin$ denote suitable
finiteness conditions.) This set-up is a step towards deeper insights into
the structure of the stable homotopy category via viewing chromatic
convergence in a categorified manner.

We then move to fibered products of model categories. Using this set-up, we
show that the category of symmetric spectra is Quillen equivalent to
the homotopy limit model structure of the left Quillen presheaf for Bousfield
arithmetic squares of spectra.

As a final application we focus on a correspondence between the homotopy
fiber of a left Bousfield localization $\C \to L_\mathcal{S}\C$ and certain right
Bousfield localizations. This is then used, among other examples, to
understand the layers of the Postnikov towers established earlier, and to
study the correspondence between stable localizations and stable
colocalizations.

\bigskip

\noindent{\bf Acknowledgements.} The first author would like to thank Dimitri
Ara for many useful conversations on some of the topics of this paper, and
Ieke Moerdijk for suggesting the idea of studying towers of localized model
structures. The second author would like to thank David Barnes for motivating
discussions and the Radboud Universiteit Nijmegen for their hospitality. Both authors
thank the referee and the associate editor for very useful comments that helped improving
the contents and presentation of the paper.

\section{Model structures for sections of Quillen presheaves}
In this section we recall the injective model structure on the category of
sections of diagrams of model categories. We will state the existence of this
model structure in general, although we will be mainly interested in the
cases of sections of towers and fibered products of model categories. Details
about these model structures can be found in~\cite[Section 2, Application
II]{Bar10}, \cite{Ber11}, \cite{Ber12}, \cite[Section 3]{GS13} and
\cite[Section 4]{Toe06}.

Let $\mathcal{I}$ be a small category. A \emph{left Quillen presheaf on
$\mathcal{I}$} is a presheaf of categories $F\colon \mathcal{I}^{{\rm op}}\to
\CAT$ such that for every $i$ in $\mathcal{I}$ the category $F(i)$ has a
model structure, and for every map $f\colon i\to j$ in $\mathcal{I}$ the
induced functor $f^*\colon F(j)\to F(i)$ has a right adjoint and they form a
Quillen pair.

\begin{definition}
A \emph{section} of a left Quillen presheaf $F\colon \mathcal{I}^{\rm op}\to
\CAT$ consists of a tuple $X=(X_i)_{i\in\mathcal{I}}$, where each $X_i$ is in
$F(i)$, and, for every morphism $f\colon i\to j$ in $\mathcal{I}$, a morphism
$\varphi_f\colon f^* X_j\to X_i$ in $F(i)$ such that the diagram
$$
\xymatrix{
(g\circ f)^*X_k\ar[r]^-{\varphi_{g\circ f}}\ar[d]_{f^*\varphi_g} & X_i \\
f^*X_j \ar[ur]_{\varphi_f}&
}
$$
commutes for every pair of composable morphisms $f\colon i\to j$ and $g\colon
j\to k$.

A \emph{morphism of sections} $\phi\colon (X,\varphi)\to (Y,\varphi')$ is
given by morphisms $\phi_i\colon X_i\to Y_i$ in $F(i)$ such that the diagram
$$
\xymatrix{
f^* X_j \ar[r]^{f^*\phi_j}\ar[d]_{\varphi_f} & f^* Y_j\ar[d]^{\varphi'_f} \\
X_i \ar[r]_{\phi_i} & Y_i
}
$$
commutes for every morphism $f\colon i\to j$ in $\mathcal{I}$.

A section $(X, \varphi)$ is called \emph{homotopy cartesian} if for every
$f\colon i\to j$ the morphism $\varphi_f\colon f^*Q_jX_j\to X_i$ is a weak
equivalence in $F(i)$, where $Q_j$ denotes a cofibrant replacement functor in $F(j)$.

\end{definition}

Recall that a model category is \emph{left proper} if pushouts of weak
equivalences along cofibrations are weak equivalences, and \emph{right
proper} if pullbacks of weak equivalences along fibrations are weak
equivalences. A model category is \emph{proper} if it is left and right
proper.

The category of sections admits an
\emph{injective} model structure, which is left or right proper, if the involved model structures are left or
right proper, respectively. A proof of the following statement can be found in \cite[Theorem~2.30 and Propostion~2.31]{Bar10}.
Recall that a model category is called
\emph{combinatorial} if it is cofibrantly generated and the underlying
category is locally presentable. Foundations of the theory of combinatorial
model categories may be found in \cite{Beke}, \cite{Dug01} and \cite{Lu}. The
essentials of the theory of locally presentable categories can be found in
\cite{AR}, \cite{GU} or~\cite{MP}.

\begin{theorem}[Barwick]\label{thm:sect_model_str}
Let $F\colon \mathcal{I}^{\rm op}\to \CAT$ be a left Quillen presheaf such
that $F(i)$ is combinatorial for every $i$ in $\mathcal{I}$. Then there
exists a combinatorial model structure on the category of sections of $F$,
denoted by $\Sect(\mathcal{I}, F)$ and called the \emph{injective model
structure}, such that a morphism of sections $\phi$ is a weak equivalence or
a cofibration if and only if $\phi_i$ is a weak equivalence or a cofibration in $F(i)$
for every $i$ in $\mathcal{I}$, respectively. Moreover, if $F(i)$ is left or
right proper for every $i\in \mathcal{I}$, then so is the model structure on
$\Sect(\mathcal{I}, F)$. \qed
\end{theorem}

Now, in order to model the homotopy limit of a left Quillen
presheaf, we would like to construct a model structure on the category of
sections whose cofibrant objects are precisely the levelwise cofibrant
homotopy cartesian sections. This will be done by taking a right Bousfield
localization of  $\Sect(\mathcal{I}, F)$. The resulting model structure will
be called the \emph{homotopy limit model structure}.

The existence of the homotopy limit model structure when the category $\Sect(\mathcal{I}, F)$
is right proper was proved in \cite[Theorem 3.2]{Ber12}. Without any properness assumptions,
the homotopy limit model structure exists as a \emph{right} model
structure, as proved in~\cite[Theorem~5.25]{Bar10}.  It follows directly from
those results that if $F(i)$ is right proper for every $i$ in $\mathcal{I}$,
then we get a \emph{full} model structure. For the reader's convenience we
spell this out in a little more detail.

\begin{theorem}
\label{thm:holim_model_str} Let $F\colon \mathcal{I}^{\rm op}\to \CAT$ be a
left Quillen presheaf such that $F(i)$ is right proper and combinatorial for
every $i$ in $\mathcal{I}$. Then there exists a combinatorial model structure
on the category of sections of $F$, called the \emph{homotopy limit model
structure}, with the same fibrations as $\Sect(\mathcal{I}, F)$ and whose
cofibrant objects are the sections that are cofibrant in $\Sect(\mathcal{I},
F)$ and homotopy cartesian.
\end{theorem}

\begin{proof}
Let $\mathcal{D}$ be the full subcategory of $\Sect(\mathcal{I}, F)$
consisting of the homotopy cartesian sections. Consider the functor
$$
\Phi\colon \Sect(\mathcal{I}, F)\longrightarrow \prod_{f\colon i\to j}\arr(F(i))
$$
defined as $\Phi((X_i)_{i\in\mathcal{I}})=\prod_{f\colon i\to j}\varphi_f$,
where $f$ runs over all morphisms of $\mathcal{I}$ and ${\arr}(-)$ denotes
the category of arrows, and let $Q$ denote an accessible cofibrant replacement
functor in $\Sect(\mathcal{I}, F)$.

The categories $\Sect(\mathcal{I}, F)$ and $\prod_{f\colon i\to j}\arr(F(i))$
are accessible (in fact, they are locally presentable; see~\cite[Corollary
1.54]{AR}) and the functor $\Phi$ is an accessible functor since it preserves
all colimits (as these are computed levelwise). Hence $\Phi$ is an accessible
functor between accessible categories.

Each $F(i)$ is combinatorial for every $i$ in $\mathcal{I}$, and hence by
\cite[Corollary A.2.6.6]{Lu} the subcategory of weak equivalences
$\weq(F(i))$ is an accessible and accessibly embedded subcategory of
$\arr(F(i))$. Therefore,	 $\prod_{f\colon i\to j}\weq(F(i))$ is an
accessible and accessibly embedded subcategory of $\prod_{f\colon i\to
j}\arr(F(i))$. By~\cite[Remark 2.50]{AR}, the preimage
$(\Phi\circ Q)^{-1}(\prod_{f\colon i\to j}\weq(F(i)))$ is an accessible and accessibly
embedded subcategory of $\Sect(\mathcal{I}, F)$. But this preimage is
precisely $\mathcal{D}$.

Now, since $\mathcal{D}$ is accessible there exists a set $\mathcal{K}$ of objects and a
regular cardinal $\lambda$ such that every object of $\mathcal{D}$ is a
$\lambda$-filtered colimit (and hence a homotopy colimit if we choose
$\lambda$ big enough; see~\cite[Proposition 7.3]{Dug01}) of objects in
$\mathcal{K}$. Moreover, since $\mathcal{D}$ is accessibly embedded this
homotopy colimit lies in $\mathcal{D}$.

The homotopy limit model structure is then the right Bousfield localization
$R_{\mathcal{K}}\Sect(\mathcal{I}, F)$. (We can perform this right Bousfield
localization because every $F(i)$ and hence $\Sect(\mathcal{I}, F)$ are
right proper.) The fact that the cofibrant objects of this new model
structure are precisely the  levelwise cofibrant homotopy cartesian sections
follows from~\cite[Theorem 5.1.5]{Hir03}.
\end{proof}

\section{Towers of model categories}\label{sec:towers}
Let $\mathbb{N}$ be the category $0\to 1\to 2\to\cdots$. A~\emph{tower of
model categories} is a left Quillen presheaf $F\colon\mathbb{N}^{\rm op}\to
\CAT$. The objects of the category of sections are then sequences $X_0,
X_1,\ldots, X_n,\ldots$, where each $X_i$ is an object of $F(i)$, together
with morphisms $\varphi_i\colon f^*X_{i+1}\to X_i$ in $F(i)$ for every $i\ge
0$, where $f\colon i\to i+1$ is the unique morphism from $i$ to $i+1$ in
$\mathbb{N}$. A morphism between two sections $\phi_{\bullet}\colon
X_{\bullet}\to Y_{\bullet}$ consists of morphisms  $\phi_i\colon X_i\to Y_i$
in $F(i)$ such that the diagram
$$
\xymatrix{
f^*X_{i+1}\ar[r] \ar[d]_{f^*\phi_{i+1}} & X_i \ar[d]^{\phi_i}\\
f^*Y_{i+1}\ar[r] & Y_i
}
$$
commutes for every $i\ge 0$.

\begin{proposition}
Let $F\colon \mathbb{N}^{\rm op}\to \CAT$ be a tower of model categories,
where $F(i)$ is a combinatorial model category for every $i\ge 0$. There
exists a combinatorial model structure on the category of sections, denoted by
$\Sect(\mathbb{N}, F)$, where a map $\phi_{\bullet}$ is a weak
equivalence or a cofibration if and only if for every $i\ge 0$ the map $\phi_i$ is a weak
equivalence or a cofibration in $F(i)$, respectively. The fibrations are the
maps $\phi_\bullet\colon X_\bullet\to Y_\bullet$ such that $\phi_0$ is a
fibration in $F(0)$ and
$$
X_{i+1}\longrightarrow Y_{i+1}\times_{f_*Y_{i}}f_*X_{i}
$$
is a fibration in $F(i+1)$ for every $i\ge 0$, where $f_*$ denotes the right
adjoint to $f^*$. The fibrant objects are those sections $X_\bullet$ such
that $X_i$ is fibrant in $F(i)$ and the morphism
$$
X_{i+1}\longrightarrow f_* X_i
$$
is a fibration in $F(i+1)$ for every $i\ge 0$. \label{prop:model_str_tow}
\end{proposition}
\begin{proof}
The existence of the required model structure follows from
Theorem~\ref{thm:sect_model_str}. The description of the fibrations follows
from~\cite[Theorem~3.1]{GS13}.
\end{proof}

\begin{proposition}
\label{prop:post_model} Let $F\colon\mathbb{N}^{\rm op}\to \CAT$ be a tower
of model categories, where each $F(i)$ is combinatorial and right proper for
every $i\ge 0$. Then there is a model structure $\Tow(F)$ on the category of
sections of $F$ with the following properties:
\begin{itemize}
\item[{\rm (i)}] A morphism $\phi_\bullet$ is a fibration in $\Tow(F)$ if and only
    $\phi_\bullet$ is a fibration in $\Sect(\mathbb{N}, F)$.
\item[{\rm (ii)}] A section $X_\bullet$ is cofibrant in $\Tow(F)$ if and only if
    $X_i$ is cofibrant in $F(i)$ and the morphism $f^*X_{i+1} \to X_{i}$
    is a weak equivalence in $F(i)$ for every~$i\ge 0$.
\item[{\rm (iii)}] A morphism $\phi_{\bullet}$ between cofibrant sections
    is a weak equivalence in $\Tow(F)$ if and only if $\phi_i$ is a weak
    equivalence in $F(i)$ for every $i\ge 0$.
\end{itemize}
\label{prop:model_str_hotow}
\end{proposition}
\begin{proof}
The existence of the model structure $\Tow(F)$ follows from
Theorem~\ref{thm:holim_model_str} applied to the left Quillen presheaf $F$.
The characterization of the weak equivalences between cofibrant objects
follows since $\Tow(F)$ is a right Bousfield localization of
$\Sect(\mathbb{N}, F)$.
\end{proof}

\subsection{Postnikov sections of model structures}\label{sec:postnikovtowers}
Let $\C$ be a left proper combinatorial model category and $n\ge 0$. The
model structure $P_n\C$ of $n$-types in $\C$ is the left Bousfield localization of $\C$ with respect to the
set of morphisms $I_{\C}\square f_n$. Here $I_{\C}$ is the set of generating
cofibrations of $\C$, $f_n\colon S^{n+1}\to D^{n+2}$ is the inclusion of
simplicial sets from the $(n+1)$-sphere to the $(n+2)$-disk, and $\square$
denotes the pushout-product of morphisms constructed using the action of
simplicial sets on $\C$ coming from the existence of framings;
see~\cite[Section 5.4]{Hov99}. A longer account about model structures for $n$-types can
be found in~\cite[Section 3]{GR14a}.

For every $n<m$ the identity is a left Quillen functor $P_m\C\to P_n\C$. Thus
we have a tower of model categories $P_{\bullet}\C\colon \mathbb{N}^{\rm
op}\to \CAT$. The objects $X_{\bullet}$ of the category of sections are
sequences
$$
\cdots\longrightarrow X_n\longrightarrow\cdots\longrightarrow X_2\longrightarrow X_1\longrightarrow X_0
$$
of morphisms in $\C$,  and its morphisms $f_\bullet\colon X_\bullet\to
Y_\bullet$ are given by commutative ladders
$$
\xymatrix{
\cdots\ar[r] & X_n\ar[r]\ar[d]^{f_n} & \cdots\ar[r] & X_2\ar[r]\ar[d]^{f_2} & X_1 \ar[r]\ar[d]^{f_1} &
X_0\ar[d]^{f_0} \\
\cdots\ar[r] & Y_n\ar[r] & \cdots\ar[r] & Y_2\ar[r]& Y_1 \ar[r] & Y_0.
}
$$

By Proposition~\ref{prop:model_str_tow}, if $\C$ is a left proper
combinatorial model category, then there exists a left proper combinatorial
model structure on the category of sections $\Sect(\mathbb{N},
P_{\bullet}\C)$, where a map $f_{\bullet}$ is a weak equivalence or a
cofibration if for every $n\ge 0$ the map $f_n$ is a weak equivalence or a
cofibration in $P_n\C$, respectively. The fibrations are the maps
$f_\bullet\colon X_\bullet\to Y_\bullet$ such that $f_0$ is a fibration in
$P_0\C$ and
$$
X_n\longrightarrow Y_n\times_{Y_{n-1}}X_{n-1}
$$
is a fibration in $P_n\C$ for every $n\ge 1$. The fibrant objects can be
characterized as follows:
\begin{lemma}
Let $X_\bullet$ be a section of $P_\bullet\C$. The following are equivalent:
\begin{itemize}
\item[{\rm (i)}] $X_\bullet$ is fibrant in $\Sect(\mathbb{N},
    P_{\bullet}\C)$.
\item[{\rm (ii)}] $X_0$ is fibrant in $P_0\C$ and $X_{n+1}\to X_{n}$ is a
    fibration in $P_{n+1}\C$ for all $n\ge 0$.
\item[{\rm (iii)}] $X_n$ is fibrant in $P_n\C$ and $X_{n+1}\to X_{n}$ is
    a fibration in $\C$ for all $n\ge 0$.
\end{itemize}
\end{lemma}
\begin{proof}
This follows because a fibration in $P_n\C$ is also a fibration in
$P_{n+1}\C$ as well as a fibration in $\C$.
\end{proof}

If the model structures for $n$-types $P_n\C$ are right proper for every
$n\ge 0$, then by Proposition~\ref{prop:model_str_hotow} the model structure
$\Tow(P_\bullet\C)$ exists and will be denoted by $\Post(\C)$. It has the
following properties:
\begin{itemize}
\item[{\rm (i)}] A morphism $f_\bullet$ is a fibration in $\Post(\C)$ if and only if
    $f_\bullet$ is a fibration in $\Sect(\mathbb{N},
    P_{\bullet}\C)$.
\item[{\rm (ii)}] A section $X_\bullet$ is cofibrant if and only if $X_n$ is
    cofibrant in $\C$ and $X_{n+1} \to X_{n}$ is a weak equivalence in
    $P_{n}\C$ for every $n\ge 0$.
\item[{\rm (iii)}] A morphism $f_{\bullet}$ between cofibrant sections is
    a weak equivalence if and only if $f_n$ is a weak equivalence in
    $P_n\C$ for every $n\ge 0$.
\end{itemize}

For every $n\ge 0$ the identity functors form a Quillen pair ${\rm
id}:\C\rightleftarrows P_n\C:{\rm id}$, since $P_n\C$ is a left
Bousfield localization of $\C$. This extends to a Quillen pair
$$
\xymatrix@C-=0.5cm{
{\rm id}:\C^{\mathbb{N}^{\rm op}}_{\rm inj} \ar@<3pt>[r] & \ar@<1pt>[l]
\Sect(\mathbb{N}, P_\bullet\C):{\rm id},
}
$$

where $\C^{\mathbb{N}^{\rm op}}_{\rm inj}$ denotes the category of
$\mathbb{N}^{\rm op}$-indexed diagrams with the injective model structure.
Indeed weak equivalences and cofibrations in $\C^{\mathbb{N}^{\rm op}}_{\rm
inj}$ are defined levelwise and every weak equivalence in $\C$ is a weak
equivalence in $P_n\C$ for all $n\ge 0$. Hence, there is a Quillen pair
$$
\xymatrix@C-=0.5cm{
\C \ar@<3pt>[rr]^-{\rm const} & & \C^{\mathbb{N}^{\rm op}}_{\rm inj} \ar@<3pt>[rr]^-{\rm id} \ar@<1pt>[ll]^-{\lim}
& & \ar@<1pt>[ll]^-{\rm id} \Sect(\mathbb{N}, P_\bullet\C) \ar@<3pt>[rr]^-{\rm id}& & \Post(\C) \ar@<1pt>[ll]^-{\rm id},
}
$$
where ${\rm const}$ denotes the constant diagram functor.

\begin{lemma}
\label{lem:post_quillen_pair} The adjunction ${\rm const}:
\C\rightleftarrows \Post(\C): \lim$ is a Quillen pair.
\end{lemma}
\begin{proof}
By  \cite[Proposition 8.5.4(2)]{Hir03}, it is enough to check that the left adjoint preserves trivial cofibrations
and cofibrations between cofibrant objects. If $f$ is a trivial cofibration
in $\C$ then ${\rm const}(f)$ is a trivial cofibration in
$\Sect(\mathbb{N}, P_\bullet\C)$. But since $\Post(\C)$ is a right
Bousfield localization of $\Sect(\mathbb{N}, P_\bullet\C)$ it has
the same trivial cofibrations. Hence ${\rm const}(f)$ is a trivial
cofibration in $\Post(\C)$.

Let $f\colon X\to Y$ be a cofibration between cofibrant objets in $\C$. Then
${\rm const}(f)$ is a cofibration between cofibrant objects in
$\Sect(\mathbb{N}, P_\bullet\C)$. But ${\rm const}(X)$ and ${\rm
const(Y)}$ are both cofibrant in $\Post(\C)$ by
Proposition~\ref{prop:post_model}. Hence ${\rm const}(f)$ is a cofibration in
$\Post(\C)$ if and only if it is a cofibration in $\Sect(\mathbb{N},
P_\bullet\C)$ (see \cite[Proposition 3.3.16(2)]{Hir03}).
\end{proof}

Let $\sset_*$ denote the category of pointed simplicial sets with the Kan--Quillen
model structure. Then the model structure $\Post(\sset_*)$
exists, since $P_n\sset_*$ is right proper for every $n\ge 0$;
see~\cite[Theorem 9.9]{Bou}.
\begin{theorem} The Quillen pair ${\rm const}:\sset_*\rightleftarrows \Post(\sset_*):\lim$ is a
Quillen equivalence.
\end{theorem}
\begin{proof}
By \cite[Proposition 1.3.13]{Hov99} it suffices to check that the derived
unit and counit are weak equivalences. Let $X$ be a fibrant simplicial set.
Then ${\rm const}(X)$ is cofibrant in $\Post(\sset_*)$, since ${\rm const}$ is
a left Quillen functor. Let
$$
\cdots\longrightarrow X_n\longrightarrow\cdots\longrightarrow X_2\longrightarrow X_1\longrightarrow X_0
$$
be a fibrant replacement of ${\rm const}(X)$ in $\Post(\sset_*)$. Hence we have
that $X_n$ is fibrant in $P_n\sset_*$ and $X_{n+1}\to X_n$ is a fibration in
$\sset_*$ and a weak equivalence in $P_n\sset_*$ for all $n\ge 0$. By
\cite[Ch.VI, Theorem 3.5]{GJ99}, the map $X\to \lim X_{\bullet}$ is a weak
equivalence.

Now, let $X_{\bullet}$ be any fibrant and cofibrant object in $\Post(\sset_*)$.
We have to see that the map ${\rm const}(\lim X_{\bullet})\to X_{\bullet}$ is
a weak equivalence in $\Post(\sset_*)$. This is equivalent to seeing that the
map $\lim X_{\bullet}\to X_n$ is a weak equivalence in $P_n\sset_*$ for every
$n\ge 0$. First note that since the category $\mathbb{N}^{\rm
op}_{>n}=\cdots\to n+3\to n+2\to n+1$ is homotopy left cofinal in
$\mathbb{N}^{\rm op}$ we have that $\lim X_{\bullet}$ is weakly equivalent to
$\lim_{\mathbb{N}^{\rm op}_{> n}} X_{\bullet}$ for every~$n$ (see
\cite[Theorem 19.6.13]{Hir03}). Hence it is enough to check that the map
$\lim_{\mathbb{N}^{\rm op}_{>n}} X_{\bullet}\to X_n$ is a weak equivalence in
$P_n\sset_*$ for all $n\ge 0$. For every $n\ge 0$ we have a map of towers
$$
\xymatrix{
\cdots\ar[r] & X_{m}\ar[r]\ar[d] & \cdots\ar[r] & X_{n+3}\ar[r]\ar[d] & X_{n+2} \ar[r]\ar[d] &
X_{n+1}\ar@{=}[d] \\
\cdots\ar@{=}[r] & X_{n+1}\ar@{=}[r] & \cdots\ar@{=}[r] & X_{n+1}\ar@{=}[r]& X_{n+1} \ar@{=}[r] & X_{n+1},
}
$$
where each vertical map is a weak equivalence in $P_{n+1}\sset_*$. Using
the Milnor exact sequence (see \cite[Ch.VI, Proposition 2.15]{GJ99}) we get a
morphism of short exact sequences
$$
\xymatrix{
0 \ar[r] & \lim\nolimits^{1}_{\mathbb{N}^{\rm op}_{>n}}\pi_{i+1}X_{\bullet}\ar[r]\ar[d] &
\pi_i(\lim\nolimits_{\mathbb{N}^{\rm op}_{> n}}X_{\bullet}) \ar[r]\ar[d] & \lim\nolimits_{\mathbb{N}^{\rm op}_{> n}}\pi_i X_{\bullet}\ar[r]\ar[d] & 0 \\
0 \ar[r] & \lim\nolimits^{1}_{\mathbb{N}^{\rm op}_{>n}}\pi_{i+1}X_{n+1}\ar[r] &
\pi_i(\lim\nolimits_{\mathbb{N}^{\rm op}_{> n}}X_{n+1}) \ar[r] & \lim\nolimits_{\mathbb{N}^{\rm op}_{> n}}\pi_i X_{n+1}\ar[r] & 0.
}
$$
For $0\le i<n$ the left and right vertical morphisms are isomorphisms, hence the map
$\lim_{\mathbb{N}^{\rm op}_{>n}} X_{\bullet}\to X_{n+1}$ is a weak
equivalence in $P_{n}\sset_*$. Therefore the map
$$
\lim\nolimits_{\mathbb{N}^{\rm op}_{>n}} X_{\bullet}\longrightarrow X_{n+1}\longrightarrow X_n
$$
is a weak equivalence in $P_{n}\sset_*$ for $n\ge 0$.
\end{proof}

\begin{corollary}
Let $X\to Y$ be a map in $\Post(\sset_*)$. Then $X\to Y$ is a weak equivalence if and only if
$\lim \widehat X\to \lim \widehat Y$ is a weak equivalence in $\sset_*$, where $\widehat X$ and $\widehat Y$ denote a fibrant replacement of $X$ and $Y$, respectively.
\end{corollary}
\begin{proof}
We have the following commutative diagram
$$
\xymatrix{
{\rm const}(\lim \widehat X) \ar[d]_g \ar[r]^-{\simeq} & \widehat X\ar[d] & \ar[l]_-{\simeq} X \ar[d]^f\\
{\rm const}(\lim \widehat Y) \ar[r]^-{\simeq} & \widehat Y & Y \ar[l]_-{\simeq}.
}
$$
The horizontal arrows are weak equivalences because they are either a fibrant replacement or because the Quillen pair ${\rm const}$ and $\lim$ is a Quillen equivalence. So $f$ is a weak equivalence if and only if $g$ is a weak equivalence. But since ${\rm const}$ preserves and reflects weak equivalences between cofibrant objects (because it is the left adjoint of a Quillen equivalence), it follows that $g$ is a weak equivalence if and only if $\lim \widehat X\to \lim \widehat Y$ is a weak equivalence.
\end{proof}

\subsection{Chromatic towers of localizations}

We can also use the homotopy limit model structure on towers of categories to
obtain a categorified version of yet another classical result. The chromatic
convergence theorem states that for a finite $p$-local spectrum $X$,
\[
X \simeq\holim_n L_nX
\]
where $L_n$ denotes left localization at the chromatic homology theory
$E(n)$ at a fixed prime $p$; see \cite[Theorem 7.5.7]{Rav92}. The prime $p$ is traditionally omitted from notation. We will see that the Quillen
adjunction between spectra and the left Quillen presheaf of chromatic
localizations of spectra induces an adjunction between the homotopy category
of finite spectra and the homotopy category of chromatic towers subject to a
suitable finiteness condition. The chromatic convergence theorem then shows
that the derived unit of this adjunction is a weak equivalence. By $\Sp$ in this section we always mean the category of \emph{$p$-local spectra symmetric spectra}~\cite{HSS} and the prime~$p$ will be fixed throughout the section.

Recall from~\cite[Section 6.1]{Hov99} that the homotopy category of a pointed model category supports
a suspension functor with a right adjoint loop functor defined via framings.
A model category is called \emph{stable} if it is pointed and the suspension and loop operators are inverse equivalences on the homotopy category. Every combinatorial stable model category admits an enrichment over the category of symmetric spectra via stable frames; see~\cite{Dug06} and \cite{Len12}.

Let $\C$ be a proper and combinatorial stable model
category. Given a prime~$p$, we define $L_n\C$ to be the left Bousfield localization of $\C$ with
respect to the $E(n)$\nobreakdash-equi\-va\-len\-ces, where $E(n)$ is considered at the prime $p$. By this, we mean Bousfield localisation at the set $I_\C \square \mathcal{S}_{E(n)}$, where $I_\C$ is the set of generating cofibrations of $\C$ and $\mathcal{S}_{E(n)}$ the generating acyclic cofibrations of $L_{E(n)}\Sp=L_n\Sp$. (The square denotes the pushout-product.) This defines a left Quillen presheaf
$$L_\bullet\C\colon \mathbb{N}^{op} \longrightarrow \CAT.$$
By Proposition \ref{prop:model_str_tow} we get the following.

\begin{proposition}
There is a left proper, combinatorial and stable model structure on the
category of sections $\Sect(\mathbb{N}, L_\bullet\C),$ such that a map
is a weak equivalence or a cofibration if and only if each
$$
f_n\colon X_n \longrightarrow Y_n
$$
is a weak equivalence or a cofibration in $L_n\C$, respectively. A map
$f_n\colon X_n \to Y_n$ is a fibration if and only if $f_0$ is a fibration in
$L_0\C$ and $$X_{n+1} \longrightarrow Y_{n+1} \times_{Y_n} X_n$$ is a
fibration in $L_{n+1}\C$ for all $n \ge 1$. \qed
\end{proposition}

Note that the resulting model structure is stable as each $L_n\C$ is stable.
We then perform a right Bousfield localization to obtain the homotopy limit
model structure. Note that this again results in a stable model
category~\cite[Proposition 5.6]{BR14} as this right localization is stable in
the sense of \cite[Definition 5.3]{BR14}. As left localization with respect
to $E(n)$ is also stable in the sense of \cite[Definition 4.2]{BR14}, $L_n\C$
is both left and right proper if~$\C$ is; see~\cite[Propositions 4.6 and
4.7]{BR14}. Hence, Proposition~\ref{prop:post_model} implies the following result.

\begin{proposition}
Let $\C$ be a proper, combinatorial and stable model category. There is a
model structure $\Chrom(\C)$ on $\Sect(\mathbb{N}, L_\bullet\C)$ with
the following properties.
\begin{itemize}
\item[{\rm (i)}] A morphism is a fibration in $\Chrom(\C)$ if and only if
    it is a fibration in $\Sect(\mathbb{N}, L_\bullet\C)$.
\item[{\rm (ii)}] An object $X_\bullet$ is cofibrant in $\Chrom(\C)$ if and only if
    all the $X_n$ are cofibrant in $\C$ and $X_{n+1} \to X_n$ is an
    $E(n)$-equivalence for each $n$.\qed
\end{itemize}
\end{proposition}

The following is useful to justify the name ``homotopy limit model
structure''. Recall that $\Sp$ denotes here the category of $p$-local spectra.

\begin{lemma}\label{prop:htpyweq}
Let $f\colon X_\bullet \to Y_\bullet$ be a weak equivalence in $\Chrom(\Sp)$.
Then $$\holim X_\bullet \longrightarrow \holim Y_\bullet$$ is a weak
equivalence of spectra.
\end{lemma}

\begin{proof}
Let $f\colon X_\bullet \to Y_\bullet$ be a weak equivalence in $\Chrom(\Sp)$.
This implies that
\[
\Ho(\Chrom(\Sp))({\rm const}(A), X_\bullet)\longrightarrow \Ho(\Chrom(\Sp))({\rm const}(A), Y_\bullet)
\]
is an isomorphism for all cofibrant $A \in \Sp$. By Lemma
\ref{lem:post_quillen_pair}, $({\rm const}, \lim)$ is a Quillen pair, so the
above is equivalent to the claim that
\[
[A, \holim X_\bullet] \longrightarrow [A, \holim Y_\bullet]
\]
is an isomorphism for all cofibrant $A \in \C$, where the square brackets
denote morphisms in the stable homotopy category. But as the class of all
cofibrant spectra detects isomorphisms in the stable homotopy category, this
is equivalent to $$\holim X_\bullet \longrightarrow \holim Y_\bullet$$ being
a weak equivalence of spectra as desired.
\end{proof}

\begin{rmk}\label{rem:generators}
It is important to note that we do not know if the converse is true. Looking
at the proof of this lemma, we see that the following are equivalent:
\begin{itemize}
\item[{\rm (i)}] There is a set of objects of the form ${\rm const}(G)$ in
    $\Chrom(\Sp)$ that detect weak equivalences.
\item[{\rm (ii)}] The weak equivalences in $\Chrom(\Sp)$ are precisely the
    $\holim$-isomorphisms.
\end{itemize}
Unfortunately, it is not known from the definition of the homotopy limit
model structure whether any of those equivalent conditions hold.
\end{rmk}

We can now turn to the main result of this subsection. For this, we need to
specify our finiteness conditions. Recall that a $p$-local spectrum is
called \emph{finite} if it is in the full subcategory of the stable homotopy
category $\Ho(\Sp)$ which contains the sphere spectrum and is closed under
exact triangles and retracts. We denote this full subcategory by
$\Ho(\Sp)^{\fin}$.

\begin{definition}
We call a diagram $X_\bullet$ in $\Chrom(\Sp)$ \emph{finitary} if $\holim
X_\bullet$ is a finite spectrum. By $\Ho(\Chrom(\Sp))^F$ we denote the full
subcategory of the finitary diagrams in the homotopy category of
$\Chrom(\Sp)$.
\end{definition}

\begin{theorem}\label{thm:chromatic}
The Quillen adjunction ${\rm const}:\Sp\rightleftarrows
\Chrom(\Sp):\lim$ induces an adjunction
$$
\Ho(\Sp)^{\fin} \lradjunction \Ho(\Chrom(\Sp))^F
$$
and the derived unit is a weak equivalence.
\end{theorem}

\begin{proof}
Firstly, we notice that the derived adjunction
\[
\mathbb{L}{\rm const}:\Ho(\Sp) \lradjunction \Ho(\Chrom(\Sp)):\mathbb{R}\lim=\holim
\]
restricts to an adjunction
\[
\mathbb{L}{\rm const}:\Ho(\Sp)^{fin} \lradjunction
\Ho(\Chrom(\Sp))^F: \mathbb{R}\lim=\holim.
\]
By definition, the homotopy limit of each finitary diagram is assumed to be a
finite spectrum. On the other side,
$$\holim(\mathbb{L}{\rm const}(X))\simeq X$$
is exactly the chromatic convergence theorem for finite spectra. The derived
unit of the above adjunction is a weak equivalence. For a cofibrant spectrum
\[
X \longrightarrow ( \holim({\rm const}(X))=\holim_n L_nX)
\]
is again the chromatic convergence theorem.
\end{proof}

We would really like to show that the above adjunction is an equivalence of
categories, that is, that the counit is a weak equivalence, meaning that
\[
{\rm const}(\holim Y_\bullet) \longrightarrow Y_\bullet
\]
is a weak equivalence for $Y_{\bullet}$ a fibrant and cofibrant finitary
diagram in $\Chrom(\Sp)$. However, to show this we would need to know that
the weak equivalences in $\Chrom(\Sp)$ are exactly the holim-isomorphisms;
see Remark~\ref{rem:generators}. Furthermore, we would not just have to know that
$\Chrom(\Sp)$ has a constant set of generators but also that those generators
are finitary,  that is, the homotopy limit of each generator is finite.

\subsection{Convergence of towers}
Let $\C$ be a left proper combinatorial model structure such that
the model structures $P_n\C$ of $n$-types (see Section~\ref{sec:postnikovtowers}) are right proper, and hence the model structure $\Post(\C)$ exists. In this section we are going to take a closer look at what it means for a
tower in $\Post(\C)$ to converge. Recall that we have a Quillen adjunction
\[
{\rm const}: \C \lradjunction \Post(\C): \lim.
\]
The following terminology appears in \cite[Definition 5.35]{Bar10}.
\begin{definition}
The model category $\C$ is \emph{hypercomplete} if the derived left adjoint of the previous Quillen adjunction is full and faithful, that is, if the composite
\[
\Ho(\C) \xrightarrow{\mathbb{L}{\rm const}} \Ho(\Post(\C)) \xrightarrow{\holim} \Ho(\C)
\]
is isomorphic to the identity.
\end{definition}
 We have seen in Section
\ref{sec:postnikovtowers} that this is true for $\C=\sset_*$. We have also seen
in Theorem \ref{thm:chromatic} that, under a finiteness assumption, the
chromatic tower of spectra $\Chrom(\Sp)$ is hypercomplete in this sense. We
can also consider the case of left Bousfield localizations of $\sset_*$, that
is,  $\C=L_\calS \sset_*$. In general, this model category will not be
hypercomplete. Let $X$ be fibrant in $L_\calS \sset_*$, that is, fibrant as a
simplicial set and $\calS$-local. If we take a fibrant replacement of the
constant tower $\rm const(Y)$ in $\Post(L_\calS \sset_*)$, we obtain a tower
\[
({\rm const}(Y))^{\fib}=(\cdots\longrightarrow Y_n \longrightarrow Y_{n-1} \longrightarrow\cdots\longrightarrow Y_0)
\]
such that all the $Y_i$ are $\calS$-local, $Y_i$ is $P_i$-local for all $i$
and $Y_n \to Y_{n-1}$ is a weak equivalence in $P_{n-1}L_\calS \sset_*$.
However, this is not a fibrant replacement of $\rm const(Y)$ in
$\Post(\sset_*)$, unless $L_\calS$ commutes with all the localizations $P_n$.
In this case, a Postnikov tower in $L_\calS\sset_*$ is also a Postnikov tower
in $\sset_*$, and hypercompleteness holds. This would be the case for
$L_\calS=L_{MR}$ for $R$ a subring of the rational numbers $\mathbb{Q}$, but
it cannot be expected in general.

Let us recapture the classical case to get a more general insight into
hypercompleteness. For $X$ in $\sset_*$ we know that $X \to\lim_n P_nX$ is a
weak equivalence. This is equivalent to saying that for all $i$,
\[
\pi_i(X) \longrightarrow \pi_i(\lim_n P_nX)
\]
is an isomorphism of groups. But we have also seen that
\[
\pi_i(\lim_nP_nX) = \lim_n\pi_i(P_nX)
\]
as well as
\[
\pi_i(P_nX)= \left\{
\begin{array}{ccc}
\pi_i(X) &\mbox{ if } & i \le n, \\
0 & \mbox{ if } & i > n. \\
 \end{array}
 \right.
\]
Putting this together we get that, indeed, $\pi_i(\lim_nP_nX) \cong \pi_i(X)$
for all $i$. This is a special case of the following. A set of \emph{homotopy generators} for a model
category $\C$ consists of a small full subcategory $\mathcal{G}$ such that
every object of $\C$ is weakly equivalent to a filtered homotopy colimit of
objects of $\mathcal{G}$, and that by \cite[Proposition 4.7]{Dug01} every
combinatorial model category has a set of homotopy generators that can be
chosen to be cofibrant. Let $\C$ be a proper
combinatorial model category with a set of homotopy generators $\G$ and
homotopy function complex $\map_\C(-,-)$. Then, for a cofibrant $X$, the map
$X \to \holim_nP_nX$ is a weak equivalence in $\C$ if and only if
$$\map_\C(G,X) \longrightarrow \map_\C(G,\holim_nP_nX) =
\holim_n\map_\C(G,P_nX)$$ is a weak equivalence in $\sset$ for all $G \in \G$, where the equality holds by~\cite[Theorem~19.4.4(2)]{Hir03}.

So from this we can see that if we had $\map_\C(G, P_nX)\cong
P_n\map_\C(G,X)$ for all $G $ in $\G$, then we would get the desired weak
equivalence because again
\[
\pi_i \map_\C(G,P_nX) = \pi_i(P_n\map_\C(G,X)).
\]
We could also reformulate this statement by not using the full set of
generators $\G$, since we are only making use of the fact that they detect
weak equivalences.

\begin{proposition}
Let $h\G$ be a set in $\C$ that detects weak equivalences. If
\[
\map_\C(G, P_nX)\cong P_n\map_\C(G,X)
\]
for every $G$ in $h\G$, then $\C$ is hypercomplete. \qed
\end{proposition}

We can follow this through with a non-simplicial example, bounded chain
complexes of $\mathbb{Z}$-modules $\Ch_b(\mathbb{Z})$. Let us briefly recall
Postnikov sections of chain complexes, which are discussed in detail in
\cite[Section 3.4]{GR14a}. As mentioned in Section~\ref{sec:postnikovtowers},
$P_n\Ch_b(\mathbb{Z})$ is the left Bousfield localization of
$\Ch_b(\mathbb{Z})$ at
\[
W_k=I_{\Ch_b(\mathbb{Z})}\square \{f_k: S^{k+1} \longrightarrow D^{k+2}\}.
\]
The generating cofibrations of the projective model structure of
$\Ch_b(\mathbb{Z})$ are the inclusions
\[
I_{\Ch_b(\mathbb{Z})}=\{ \mathbb{S}^{n-1} \longrightarrow \mathbb{D}^n \,\,|\,\, n \ge 1\}
\]
where $\mathbb{S}^{n-1}$ is the chain complex which only contains
$\mathbb{Z}$ in degree $n-1$ and is zero in all other degrees, and
$\mathbb{D}^n$ is $\mathbb{Z}$ in degrees $n$ and $n-1$ with the identity
differential and zero everywhere else. We can thus work out that
\[
W_k=\{ \mathbb{S}^{n+k+1} \longrightarrow \mathbb{D}^{n+k+2}\,\,|\,\, n \ge 0\}.
\]
This means that a chain complex is a $k$-type if and only if its homology
vanishes in degrees $k+1$ and above. The localization $M \longrightarrow
P_kM$ is simply truncation above degree $k$.

\bigskip
Let $\Hom(M, N)$ denote the mapping chain complex for $M$, $N$ in
$\Ch_b(\mathbb{Z})$, that is,
\[
\Hom(M, N)_k=\prod_i \Hom_\mathbb{Z}(M_i, N_{i+k})
\]
with differential $(df)(x)=d(f(x))+(-1)^{k+1}f(d(x))$; see for example
\cite[Chapter ~4.2]{Hov99}. We note that
\[
\pi_i(\map_{\Ch_b(\mathbb{Z})}(M,N)) = H_i(\Hom(M, N))
\]
because
\begin{multline}
\pi_i(\map_{\Ch_b(\mathbb{Z})}(M, N))=[S^i, \map_{\Ch_b(\mathbb{Z})}(M,
N)]_{\sset_*} = [M\otimes^L S^i, N]_{\Ch_b(\mathbb{Z})}\nonumber\\ = [M[i],
N]_{\Ch_b(\mathbb{Z})}=[M\otimes \mathbb{\mathbb{Z}}[i],
N]_{\Ch_b(\mathbb{Z})} = [\mathbb{\mathbb{Z}}[i], \Hom(M,
N)]_{\Ch_b(\mathbb{Z})} \nonumber\\ = H_i(\Hom(M, N)).
\end{multline}
So $\Ch_b(\mathbb{Z})$ is hypercomplete if $\Hom(G, P_nN)$ is
quasi-isomorphic to $P_n\Hom(G, N)$ for all $G$ in $h\G$. For bounded below
chain complexes, a set that detects weak equivalences can be taken to be
\[
h\G= \{ \mathbb{S}^i=\mathbb{Z}[i] \mid i \ge 0\}.
\]
We have the following diagram of short exact sequences:
\[
\xymatrix@C-=0.5cm{ \Ext_\mathbb{Z}(H_i(M), H_{i+1}(N)) \ar[r]\ar[d] &
H_i(\Hom(M, N)) \ar[r]\ar[d] & \Hom_\mathbb{Z}(H_i(M), H_i(N)) \ar[d]\\
\Ext_\mathbb{Z}(H_i(M), H_{i+1}(P_n N)) \ar[r] & H_i(\Hom(M,P_n N))
 \ar[r] & \Hom_\mathbb{Z}(H_i(M), H_i(P_n N)).
}
\]

Using the 5-lemma we can read off that $H_i(\Hom(M, P_n N))=0$ for $i >n$ as
desired and that $$H_i(\Hom(M, P_n N))=H_i(\Hom(M, N))$$ for $i \le n-1$, but
unless $\Ext_\mathbb{Z}(H_n(M), H_{n+1}(N))=0$ we do \emph{not} get that
$$H_n(\Hom(M, P_n N))=H_n(\Hom(M, N)).$$ Note that in general it is not true that $\Hom(M, P_n N)\simeq P_n\Hom(M,
N).$ However, as we only require the case
$M=\mathbb{S}^i$, we have that
\[
\Hom(\mathbb{S}^i, N)= N[n],
\]
where $N[n]$ is the $n$-fold suspension of $N$. Thus,
\[
\Hom(G, P_n N)= P_n \Hom(G, N)
\]
for all $G$ in $h\G$, so $\Ch_b(\mathbb{Z})$ is hypercomplete as expected.

\begin{rmk}
Another important example of a tower of model structures occurring in nature is given by
the Taylor tower of Goodwillie calculus, where for every $n$ one considers the $n$-excisive model structure on the category of small endofunctors of simplicial sets; see~\cite[Section 4]{BCR07}. We do not discuss this example in this paper, and detailed relations to the aforementioned references could be a topic for future research.
\end{rmk}

\section{Homotopy fibered products of model categories}
Let $\mathcal{I}$ be the small category
$$
1\stackrel{\alpha}{\longleftarrow} 0\stackrel{\beta}{\longrightarrow}  2.
$$
A \emph{pullback diagram of model categories} is a left Quillen presheaf
$F\colon \mathcal{I}^{\rm op}\to \CAT$. The objects $X_\bullet$ of the
category of sections  are given by three objects $X_0, X_1$ and  $X_2$ in
$F(0), F(1)$ and $F(2)$, respectively, together with morphisms
$$
\alpha^*X_1\longrightarrow X_0\longleftarrow \beta^*X_2
$$
in $F(0)$. A morphism $\phi_\bullet\colon X_\bullet\to Y_\bullet$ consists of
morphisms $\phi_i\colon X_i\to Y_i$ in $F(i)$ for $i=0,1,2$, such that the
diagram
$$
\xymatrix{
\alpha^*X_1 \ar[r]\ar[d]^{\alpha^*\phi_1} & X_0  \ar[d]^{\phi_0} & \beta^* X_2
\ar[l]\ar[d]^{\beta^*\phi_2} \\
\alpha^*Y_1 \ar[r] & Y_0 & \beta^*Y_2 \ar[l]
}
$$
commutes.

\begin{proposition}
Let $F\colon \mathcal{I}^{\rm op}\to \CAT$ be a pullback diagram of model
categories such that $F(i)$ is a combinatorial model category for every $i$
in $\mathcal{I}$. Then there exists a combinatorial model structure on the
category of sections $\Sect(\mathcal{I}, F)$, where a map
$\phi_{\bullet}$ is a weak equivalence or a cofibration if and only if $\phi_i$ is a weak
equivalence or cofibration in $F(i)$ for every $i$ in $\mathcal{I}$. The
fibrations are the maps $\phi_\bullet\colon X_\bullet\to Y_\bullet$ such that
$f_{0}$ is a fibration in $F(0)$ and
$$
X_1\longrightarrow Y_1\times_{\alpha_*Y_{0}}\alpha_*X_{0} \quad\mbox{and}\quad X_2
\longrightarrow Y_2\times_{\beta_*Y_{0}}\beta_*X_{0}
$$
are fibrations in $F(1)$ and $F(2)$, respectively. In particular, $X_\bullet$
is fibrant if $X_i$ is fibrant in $F(i)$ and
$$
X_1\longrightarrow \alpha_*X_{0}\quad\mbox{and}\quad X_2\longrightarrow \beta_*X_{0}
$$
are fibrations in $F(1)$ and $F(2)$, respectively \label{prop:model_str_pull}.
\end{proposition}

\begin{proof}
The existence of the required model structure follows from
Theorem~\ref{thm:sect_model_str}. The description of the fibrations follows
from~\cite[Theorem~3.1]{GS13}.
\end{proof}

\begin{proposition}
Let $F\colon \mathcal{I}^{\rm op}\to \CAT$ be a pullback diagram of model
categories such that $F(i)$ is combinatorial and right proper for every
$i$ in $\mathcal{I}$. Then there is a model structure $\Fibprod(F)$ on the
category of sections of $F$, called the \emph{homotopy fibered product model
structure}, with the following properties:
\begin{itemize}
\item[{\rm (i)}] A morphism $\phi_\bullet$ is a fibration in
    $\Fibprod(F)$ if and only if $\phi_{\bullet}$ is a fibration in
    $\Sect(\mathcal{I}, F)$.
\item[{\rm (ii)}] A section $X_\bullet$ is cofibrant in $\Fibprod(F)$ if and only if
    $X_i$ is cofibrant in $F(i)$ for every $i$ in $\mathcal{I}$ and the
    morphisms $\alpha^*X_{1} \to X_{0}$ and $\beta^*X_2\to X_{0}$ are
    weak equivalences in $F(0)$.
\item[{\rm (iii)}] A morphism $\phi_{\bullet}$ between cofibrant sections
    is a weak equivalence if and only if $\phi_i$ is a weak equivalence
    in $F(i)$ for every $i$ in $\mathcal{I}$.
\end{itemize}
\label{prop:model_str_hopull}
\end{proposition}

\begin{proof}
The existence of the model structure $\Fibprod(F)$ follows from
Theorem~\ref{thm:holim_model_str} applied to the left Quillen presheaf $F$.
The characterization of the weak equivalences between cofibrant objects
follows since $\Fibprod(F)$ is a right Bousfield localization of
$\Sect(\mathcal{I}, F)$.
\end{proof}

\subsection{Bousfield arithmetic squares of homological localizations}
Let $\C$ be a left proper combinatorial stable model category and $E$ any
spectrum. The model structure $L_E\C$ is the left Bousfield localization of $\C$ with respect to the
set $I_\C\square \mathcal{S}_E$. Here $I_\C$ is the set of
generating cofibrations of $\C$, the set $\mathcal{S}_E$ consists of the
generating trivial cofibrations of the homological localization $L_E\Sp$, and
$\square$ is the pushout-product defined via the action $\C\times \Sp\to \C$.
This model structure is an example of a left Bousfield localization along a Quillen bifunctor, as studied in~\cite{GR14a}.

Now, let $J$ and $K$ be a partition of the set of prime numbers. By
$\mathbb{Z}_J$ we denote the $J$-local integers, and by $MG$ the Moore
spectrum of the group $G$. Consider the model structures
$L_{M\mathbb{Z}_J}\C$, $L_{M\mathbb{Z}_K}\C$ and $L_{M\mathbb{Q}}\C$. Since,
for every set of primes $P$, every $M\mathbb{Z}_P$\nobreakdash-equivalence is
an $M\mathbb{Q}$-equivalence, the identities $L_{M\mathbb{Z}_J}\C\to
L_{M\mathbb{Q}}\C$ and $L_{M\mathbb{Z}_K}\C\to L_{M\mathbb{Q}}\C$ are left
Quillen functors.

Thus we have a pullback diagram of model categories $L_{\bullet}\C\colon
\mathcal{I}^{\rm op}\to \CAT$, where $\mathcal{I}=1\leftarrow 0\rightarrow 2$
and $L_{0}\C=L_{M\mathbb{Q}}\C$, $L_1\C=L_{M\mathbb{Z}_J}\C$ and
$L_2\C=L_{M\mathbb{Z}_K}\C$.

If $\C$ is a left proper combinatorial stable model category, then by
Proposition~\ref{prop:model_str_pull} the model structure
$\Sect(\mathcal{I}, L_{\bullet}\C)$ exists, and it is also a stable
model structure because each of the involved model categories is stable.

Moreover, if in addition the model structures $L_{M\mathbb{Z}_J}\C$,
$L_{M\mathbb{Z}_K}\C$ and $L_{M\mathbb{Q}}\C$ are right proper, then by
Proposition~\ref{prop:model_str_hopull} the model structure
$\Fibprod(L_\bullet\C)$, which we denote by $\Bou(\C)$, also exists. The model
structure $\Bou(\C)$ is also stable, since it is a right Bousfield
localization with respect to a set of stable objects;
see~\cite[Proposition~5.6]{BR14}.

\begin{lemma}
The adjunction ${\rm const}: \C\rightleftarrows \Bou(\C): \lim$ is
a Quillen pair.
\end{lemma}
\begin{proof}
The proof is the same as the one for Lemma~\ref{lem:post_quillen_pair}.
\end{proof}

Note that for any spectrum $E$, the model
structure $L_E\Sp$ is right proper~\cite[Proposition 4.7]{BR14}, hence the
model structure $\Bou(\Sp)$ exists.

\begin{theorem}\label{thm:Bousfieldsquare}
Let $\C$ be a proper and combinatorial stable model category. The Quillen pair ${\rm const}:\C\rightleftarrows \Bou(\C):\lim$
is a Quillen equivalence.
\end{theorem}

\begin{proof}
By \cite[Proposition 1.3.13]{Hov99} it suffices to check that the derived
unit and counit are weak equivalences.

Let $X$ be a fibrant and cofibrant object in $\C$. We need to show that
\[
X \longrightarrow \lim({\rm const}(X)^{\fib})
\]
is a weak equivalence in $\C$, where $(-)^{\fib}$ denotes a fibrant
replacement in $\Bou(\C)$. The constant diagram ${\rm const}(X)$ is
cofibrant in $\Bou(\C)$ since ${\rm const}$ is a left Quillen functor. Let
$$
L_{M\mathbb{Z}_J}X\longrightarrow L_{M\mathbb{Q}}X \longleftarrow L_{M\mathbb{Z}_K}X
$$
be a fibrant replacement of ${\rm const}(X)$ in $\Bou(\C)$. We have that
$L_{M\mathbb{Z}_K}X$, $L_{M\mathbb{Z}_J}X$ and $L_{M\mathbb{Q}}X$ are fibrant
in $L_{M\mathbb{Z}_K}\C$, $L_{M\mathbb{Z}_J}\C$ and $L_{M\mathbb{Q}}\C$,
respectively, and the two maps are fibrations in $\C$ and weak equivalences
in $L_{M\mathbb{Q}}\C$. Furthermore, the three localisations are smashing in $\Sp$, so by \cite[Lemma 6.7]{BR}
\[
L_{M\mathbb{Z}_K}X= X \wedge M\mathbb{Z}_K, L_{M\mathbb{Q}}X= X \wedge M\mathbb{Q} \,\,\,\mbox{and}\,\,\, L_{M\mathbb{Z}_J}X=X \wedge M\mathbb{Z}_J.
\]
By \cite[Proposition~2.10]{Bou79} we have that
\[
\lim(M\mathbb{Z}_K \longrightarrow M\mathbb{Q} \longleftarrow M\mathbb{Z}_J) = S,
\]
where $S$ denotes the sphere spectrum. Thus, the map
$$
X\longrightarrow \lim (L_{M\mathbb{Z}_K}X\longrightarrow L_{M\mathbb{Q}}X\longleftarrow
L_{M\mathbb{Z}_J}X) = X \wedge \lim(M\mathbb{Z}_K \longrightarrow M\mathbb{Q} \longleftarrow M\mathbb{Z}_J)
$$
is a weak equivalence. The last equality follows because homotopy pullbacks commute with the action of spectra coming from framings, since in stable categories they are equivalent to homotopy pushouts.

Now, let $X_{\bullet}$ be any fibrant and cofibrant object in $\Bou(\C)$. We
have to see that the map
$$
{\rm const}(\lim X_{\bullet})\longrightarrow X_{\bullet}
$$
is a weak equivalence in $\Bou(\C)$. This is equivalent to saying that the
map $\lim X_{\bullet}\to X_1$ is a weak equivalence in
$L_{M\mathbb{Z}_J}\C$, $\lim X_{\bullet}\to X_2$ is a weak equivalence in
$L_{M\mathbb{Z}_K}\C$ and $\lim X_{\bullet}\to X_{12}$ is a weak equivalence
in $L_{M\mathbb{Q}}\C$.

Note that if $A\to B$ is a weak equivalence in
$L_{M\mathbb{Q}}\C$, $A$ is fibrant in $L_{M\mathbb{Z}_K}\C$ and $B$ is
fibrant in $L_{M\mathbb{Q}}\C$, then $A\to B$ is a weak equivalence in
$L_{M\mathbb{Z}_J}\C$. To see this, let $A\to L_{M\mathbb{Z}_J}A$ be a
fibrant replacement of $A$ in $L_{M\mathbb{Z}_J}\C$. We are going to use \cite[Lemma 6.7]{BR} again, which says that the weak equivalences in $L_{M\mathbb{Z}_J}\C$ are morphisms $f$ in $\C$ such that
$f \wedge M\mathbb{Z}_J$ is a weak equivalence in $\C$. This makes the following argument the same as it would be for $\C = \Sp$.

 Since $B$ is fibrant in
$L_{M\mathbb{Q}}\C$, it is so in $L_{M\mathbb{Z}_J}\C$. Thus, there is a
lifting
$$
\xymatrix{
A\ar[r] \ar[d] & B \\
L_{M\mathbb{Z}_J}A. \ar@{.>}[ur]
}
$$
The left arrow is a weak equivalence in $L_{M\mathbb{Z}_J}\C$ and hence a
weak equivalence in $L_{M\mathbb{Q}}\C$. Therefore the dotted arrow is a
weak equivalence in $L_{M\mathbb{Q}}\C$ between fibrant objects in
$L_{M\mathbb{Q}}\C$. (Observe that $L_{M\mathbb{Z}_J}A$ is fibrant in
$L_{M\mathbb{Z}_J}\C$ and $L_{M\mathbb{Z}_K}\C$, and hence in
$L_{M\mathbb{Q}}\C$.) Thus, it is a weak equivalence in $\C$. This
completes the proof of the claim since weak equivalences in $\C$ are weak
equivalences in $L_{M\mathbb{Z}_J}\C$.

Since $X_\bullet$ is fibrant and cofibrant, we have that in the pullback
diagram
$$
\xymatrix{
\lim X_\bullet \ar[r]^{f_2} \ar[d]_{f_1} & X_2 \ar[d] \\
X_1 \ar[r] & X_{12}
}
$$
$X_1$, $X_2$ and $X_{12}$ are fibrant in $L_{M\mathbb{Z}_J}\C$,
$L_{M\mathbb{Z}_K}\C$ and $L_{M\mathbb{Q}}\C$, respectively, and the right
and bottom arrows are weak equivalences in $L_{M\mathbb{Q}}\C$ and fibrations in
$L_{M\mathbb{Z}_K}\C$ and $L_{M\mathbb{Z}_J}\C$, respectively. By the
previous observation and right properness of the model structures involved,
the map $f_1\colon \lim X_\bullet\to X_1$ is a weak equivalence in
$L_{M\mathbb{Z}_J},$ and $f_2\colon\lim X_\bullet\to X_2$ is a weak equivalence
in $L_{M\mathbb{Z}_K}\C$, respectively. Thus, the map $\lim X_\bullet \to
X_{12}$ is also a weak equivalence in $M\mathbb{Q}$, which means that ${\rm
const}(\lim X_{\bullet})\longrightarrow X_{\bullet}$ is an objectwise weak
equivalence, and thus a weak equivalence in $\Bou(\C)$ as claimed.
\end{proof}

\begin{rmk}
There is a higher chromatic version of the objectwise statement. Here~$\Sp$ denotes the category
of $p$-local spectra. There is a
homotopy fiber square
\[
\xymatrix{
L_n X \ar[d]\ar[r] & L_{K(n)} X \ar[d] \\
L_{n-1}X \ar[r] & L_{n-1}L_{K(n)}X;
}
\]
see \cite[Section 3.9]{Dwyer}. However, we cannot apply the methods of this
section to get a result analogously to Theorem \ref{thm:Bousfieldsquare}.
This is due to the fact that $L_{K(n)}L_{n-1}\Sp$ is trivial as a model
category. (By \cite[Theorem 2.1]{Rav84}, a spectrum is $E(n-1)$-local if and
only if it is $K(i)$-local for $1 \le i \le n-1$. But the $K(n)$-localization
of a $K(m)$-local spectrum is trivial for $n \neq m$.) Consider the homotopy
fibered product model structure on
\[
L_{n-1}\Sp \longrightarrow L_{n-1}L_{K(n)}\Sp \longleftarrow L_{K(n)}\Sp.
\]
A fibrant and cofibrant diagram
\[
X_1 \xrightarrow{f_1} X_0 \xleftarrow{f_2} X_2
\]
would have to satisfy that $X_1$ is $E(n-1)$-local and $f_1$ is an
$L_{n-1}L_{K(n)}$ localization. By the universal property of localizations,
this means that $f_1$ factors over $L_{n-1}L_{K(n)}X_1 \to X_0$.
However, as $X_1$ is $E(n-1)$-local and thus $K(n)$-acyclic, this map (and
thus $f_1$) is trivial. Thus we cannot reconstruct a pullback square like the
above from this model structure.
\end{rmk}

\subsection{Homotopy fibers of localized model categories}
We will use the homotopy fibered product model structure to describe the
homotopy fiber of Bousfield localizations. We can then use this to describe
the layers of a Postnikov tower, among other examples.

Let $\C$ be a left proper pointed combinatorial model category and let
$\calS$ be a set of morphisms in $\C$. The identity $\C\to L_\calS\C$ is a
left Quillen functor and thus we have a pullback diagram of model categories
$L^{\calS}_{\bullet}\C\colon \mathcal{I}^{\rm op}\to \CAT$, where
$\mathcal{I}=1\leftarrow 0\rightarrow 2$, and $L^\calS_0\C=L_\calS\C$,
$L^\calS_1\C=*$ and $L^\calS_2\C=\C$. (Here $*$ denotes the category with one
object and one identity morphism with the trivial model structure.)

A section of $L^{\calS}_\bullet\C$ is a diagram $*\to Y\leftarrow X$ in $\C$
where $*$ denotes the zero object. There is an adjunction
$$
\xymatrix@C-=0.5cm{
{\rm const}:\C \ar@<3pt>[r] & \ar@<1pt>[l]
\Sect(\mathcal{I}, L^{\calS}_\bullet\C):{ev_2},
}
$$
where ${\rm const}(X)=(*\to X\stackrel{1}{\leftarrow}X)$ and $ev_2(*\to
Y\leftarrow X)=X$. We will denote $\Fibprod(L_{\bullet}^{\calS})$ by
$\Hofib(L_{\bullet}^{\calS})$ and we will call it the \emph{homotopy fiber}
of the Quillen pair $\C\rightleftarrows L_{\calS}\C$.
\begin{definition}
Let $\C$ be a proper pointed combinatorial model category and let $\K$ be a
set of objects and $\calS$ be a set of morphisms in $\C$. We say that the
colocalized model structure $C_\K\C$ and the localized model structure
$L_\calS\C$ are \emph{compatible} when for every object $X$ in $\C$, $X$ is
$\K$-colocal if and only if $X$ is cofibrant in $\C$ and the map $*\to X$ is
an $\calS$-local equivalence.
\end{definition}

The stable case is discussed in detail in \cite[Section 10]{BR14} where such
model structures are called ``orthogonal''; see also Section~\ref{sec:orthogonal}.

\begin{rmk}
\label{rmk:compatibility} Note that if $C_\K\C$ and $L_\calS\C$ are
compatible, then it follows from the definitions that $*\to Y\leftarrow X$ is
cofibrant in $\Hofib(L^{\calS}_\bullet\C)$ if and only if both $X$ and $Y$ are
$\K$-colocal and cofibrant in $\C$. If $*\to Y\leftarrow X$ is moreover
fibrant in $\Hofib(L^{\calS}_\bullet\C)$, then $Y$ is weakly contractible
since $Y$ is $\calS$-local and $*\to Y$ is an $\calS$-equivalence and $X\to
Y$ is a fibration in $\C$.
\end{rmk}

\begin{theorem}
Let $\C$ be a proper pointed combinatorial model category and let $\K$ be a
set of objects and $\calS$ be a set of morphisms in $\C$. If $C_\K\C$ and
$L_\calS\C$ are compatible, then the adjunction
$$
\xymatrix@C-=0.5cm{
{\rm const}: C_\K\C \ar@<3pt>[r] & \ar@<1pt>[l]
\Hofib(L^{\calS}_\bullet\C):{ev_2}
}
$$
is a Quillen equivalence. \label{thm:hofib_Quillen_equiv}
\end{theorem}
\begin{proof}
We will first show that the adjunction is a Quillen pair. By \cite[Propostion~8.5.4(2)]{Hir03}, it is
enough to check that the left adjoint preserves trivial cofibrations and
sends cofibrations between cofibrant objects to cofibrations.

Let $f$ be a trivial cofibration in $C_\K\C$. Then $f$ is a trivial
cofibration in $\C$ and therefore ${\rm const}(f)$ is a trivial cofibration
in $\Sect(\mathcal{I}, L^{\calS}_\bullet\C)$ and thus a trivial
cofibration in $\Hofib(L^\calS_\bullet\C)$.

Now let $f\colon X\to Y$ be a cofibration between cofibrant objects in
$C_\K\C$. Then $f$ is a cofibration between cofibrant objects in $\C$ and
hence ${\rm const}(f)$ is also a cofibration between cofibrant objects in
$\Sect(\mathcal{I}, L^{\calS}_\bullet\C)$. But ${\rm const}(X)$ and
${\rm const}(Y)$ are cofibrant in $\Hofib(L^\calS_\bullet\C)$, since $C_\K\C$
and $L_\calS\C$ are compatible and therefore the maps $*\to X$ and $*\to Y$
are $\calS$-local equivalences. Hence ${\rm const}(f)$ is a cofibration in
$\Hofib(L^\calS_\bullet\C)$, by \cite[Proposition 3.3.16(2)]{Hir03}.

To prove that it is a Quillen equivalence, it suffices to show that the
derived unit and counit are weak equivalences; see \cite[Proposition
1.3.13]{Hov99}. Let $X$ be a cofibrant object in $C_\K\C$. Then we can
construct a fibrant replacement for ${\rm const}(X)$ in
$\Hofib(L^\calS_\bullet\C)$ as follows:
$$
\xymatrix{
\ast\ar[r]\ar@{=}[d] & X \ar@{=}[r] \ar@{>->}[d] & X \ar@{>->}[d] \\
\ast\ar[r] & L_\calS X & X', \ar@{->>}[l]
}
$$
where the map $X\to L_\calS X$ is a trivial cofibration in $L_\calS\C$ and
$X\to X'\to L_\calS X$ is a factorization in $\C$ of the previous map as a
trivial cofibration followed by a fibration.  Indeed, the map between the two
sections is a trivial cofibration in $\Hofib(L^\calS_\bullet\C)$ since it is
a levelwise trivial cofibration, and $*\to L_\calS X\leftarrow X'$ is fibrant
in $\Hofib(L^\calS_\bullet\C)$ since $L_\calS X$ is fibrant in $L_\calS\C$,
$X'$ is fibrant in $\C$ and $X'\to L_\calS X$ is a fibration in $\C$.

Therefore the map $X\to ev_2({\rm const}(X))\to ev_2(R({\rm const}(X)))$,
where $R$ denotes fibrant replacement in $\Hofib(L^\calS_\bullet\C)$, is
precisely the map $X\to X'$, which is a weak equivalence in $C_\K\C$ since it
was already a weak equivalence in $\C$.

Finally, let $*\to Y\leftarrow X$ be a fibrant and cofibrant section in
$\Hofib(L^\calS_\bullet\C)$. We need to check that the composite
$$
{\rm const}(Q(ev_2(*\to Y\leftarrow X)))\longrightarrow {\rm const}(ev_2(*\to Y\leftarrow X))
\longrightarrow (*\to Y\leftarrow X)
$$
is a weak equivalence in $\Hofib(L^\calS_\bullet\C)$. But $ev_2(*\to
Y\leftarrow X)=X$ is already cofibrant in $C_\K\C$, by
Remark~\ref{rmk:compatibility}. Therefore, we need to show that the map of
sections
$$
\xymatrix{
\ast\ar[r]\ar@{=}[d] & X\ar[d] \ar@{=}[r] & X \ar@{=}[d] \\
\ast \ar[r] & Y & X\ar[l]
}
$$
is a weak equivalence in $\Hofib(L^\calS_\bullet\C)$. Since both sections are
cofibrant, it is enough to see that the map in the middle is a weak
equivalence in $L_\calS\C$, which follows again from
Remark~\ref{rmk:compatibility}.
\end{proof}

\subsection{Postnikov sections and connective covers of simplicial sets}

We can use this setup to describe the ``layers'' of Postnikov towers. Let
$\sset_*$ denote the category of pointed simplicial sets. Consider the model
structure $P_k\sset_*=L_\calS\sset_*$ for $k$\nobreakdash-types,  that is,
the left Bousfield localization of $\sset_*$ with respect to the set of inclusions
$\calS=\{S^{k+1}\to D^{k+2}\}$. If $\K=\{S^{k+1}\}$, then the right Bousfield localization
$C_k\sset_*=C_{\K}\sset_*$ is the model structure for $k$-connective covers, and $P_k\sset_*$ and $C_k\sset_*$ are compatible, since for every $X$ there is a fiber sequence
$$
C_k X\longrightarrow X\longrightarrow P_k X,
$$
where $C_k X$ denotes the $k$th connective cover of $X$. By
Theorem~\ref{thm:hofib_Quillen_equiv} the model categories $C_k\sset_*$ and
$\Hofib(L^\calS_{\bullet}\sset_*)$ are Quillen equivalent.

Let $\calS=\{S^{n+1}\to D^{n+2}\}$ and $\K=\{S^{n+1}\}$, as before, and let $\C$ be a proper combinatorial model category. Then we define $L_\calS\C$ as the left Bousfield localization of $\C$ with respect to the set $I_{\C}\square\calS$ and $C_\K\C$ as the right Bousfield localization of $\C$ with respect to $\mathcal{G}_{\C}\otimes \K$. Here $I_{\C}$ is the set of generating cofibrations of $\C$, $\mathcal{G}_{\C}$ is a set of homotopy generators, $\otimes$ denotes the simplicial action given by a framing and $\square$ the pushout product. A fuller account of localized model structures along Quillen bifunctors can be found in~\cite{GR14a}. In general, $L_\calS\C$ and $C_\K\C$ are not necessarily compatible, so Theorem~\ref{thm:hofib_Quillen_equiv} will not hold in this case for arbitrary $\C$. However, examples where compatibility holds include the category of chain complexes $\Ch_b(R)$ and stable localizations; see Section \ref{sec:orthogonal}.

We can also consider $\Hofib(L^\calS_{\bullet}P_{k+1}\sset_*)$. Since for
every $X$ we have a fibration
$$
K(\pi_{k+1}X, k+1)\longrightarrow P_{k+1}X\longrightarrow P_k X,
$$
the model structures $C_k P_{k+1}\sset_*$ and $P_k P_{k+1}\sset_*=P_k
\sset_*$ are compatible. Hence Theorem~\ref{thm:hofib_Quillen_equiv} directly
implies

\begin{corollary}
The model structures $C_k P_{k+1}\sset_*$ and
$\Hofib(L^\calS_{\bullet}P_{k+1}\sset_*)$ are Quillen equivalent.\qed
\end{corollary}
This means that we can view $C_k P_{k+1}\sset_*$ as the $k$th layer of the
Postnikov tower model structure. Note that $\Ho(C_kP_{k+1}\sset_*)$ is
equivalent to the category of abelian groups for $k\ge 1$.

\subsection{Nullifications and cellularizations of spectra}\label{sec:stablecompatible}
Let $\Sp$ be a suitable model structure for the category of spectra, for
instance, symmetric spectra and let $\calS$ be a single map $E\to *$. Then
$L_\calS\Sp=P_E\Sp$ is called the \emph{$E$-nullification} of $\Sp$ and
$C_E\Sp$ is called the \emph{$E$-cellularization} of $\Sp$. As follows from
\cite[Theorem~3.6]{Gut12} we have the following compatibility between
localized and colocalized model structures:
\begin{itemize}
\item[{\rm (i)}] If the induced map ${\rm Ho}(\Sp)(\Sigma^{-1}E,C_E X)\to
    {\rm Ho}(\Sp)(\Sigma^{-1}E, X)$ is injective for every $X$, then
    $C_E\Sp$ and $P_E\Sp$ are compatible.
\item[{\rm (ii)}] If the induced map ${\rm Ho}(\Sp)(E, X)\to {\rm
    Ho}(\Sp)(E, P_{\Sigma E}X)$ is the zero map for every $X$, then
    $C_E\Sp$ and $P_{\Sigma E}\Sp$ are compatible.
\end{itemize}

\subsection{Stable localizations and colocalizations}\label{sec:orthogonal}
Let $\C$ be a proper combinatorial stable model category and let $\G_{\Sp}$
denote a set of cofibrant homotopy generators for the model category of symmetric
spectra~$\Sp$. Recall that a set of \emph{homotopy generators} for a model
category $\C$ consists of a small full subcategory $\mathcal{G}_\C$ such that
every object of $\C$ is weakly equivalent to a filtered homotopy colimit of
objects of $\mathcal{G}_\C$, and that by \cite[Proposition 4.7]{Dug01} every
combinatorial model category has a set of homotopy generators that can be
chosen to be cofibrant.

A set of maps $\calS$ in a stable model category is said to be \emph{stable} if the class
of $\calS$-local objects is closed under suspension. Let $\calS$ be a stable set of morphisms in $\C$ and let $\K={\rm
cof}(\calS)$ be the set of cofibers of the elements of $\calS$. Then we have
that  ${\rm cof}(\calS\otimes \G_{\Sp})={\rm cof}(\calS)\otimes
\G_{\Sp}=\K\otimes \G_{\Sp}$, where $\otimes$ denotes the action of $\Sp$ on $\C$. Hence, by \cite[Proposition~10.3]{BR14} it
follows that $L_{\calS\otimes \G_{\Sp}}\C$ and $C_{\K\otimes\G_{\Sp}}\C$ are
compatible. Therefore, Theorem~\ref{thm:hofib_Quillen_equiv} readily implies the following fact.
\begin{corollary}
The model categories $C_{\K\otimes\G_{\Sp}}\C$ and
$\Hofib(L^{\calS\otimes\G_{\Sp}}_{\bullet}\C)$ are Quillen equivalent. \qed
\end{corollary}

\end{document}